\documentclass[11pt,a4paper]{amsart}
\usepackage[colorlinks=true, pdfstartview=FitV, linkcolor=blue, 
citecolor=blue]{hyperref}

\usepackage{amssymb,amscd}
\usepackage{microtype,xcolor,enumerate,comment}
\usepackage{a4wide}
\usepackage {diagbox} 
\usepackage{float}

\theoremstyle{plain}
\newtheorem{theorem}{Theorem}
\newtheorem{corollary}{Corollary}
\newtheorem{lemma}{Lemma}

\newtheorem{definition}{Definition}

\theoremstyle{definition} % 定義スタイル（直立）
\newtheorem{remark}{Remark}

\begin{document}

\title[On the number of $k$-full integers between three successive 
$k$-th powers]
{On the number of $k$-full integers between three successive 
$k$-th powers}

\author{Shusei Narumi}
\author{Yohei Tachiya}
\address{Graduate School of Science and Technology,\newline
\hspace*{\parindent}Hirosaki University, Hirosaki 036-8561, Japan}
\email{h25ms125@hirosaki-u.ac.jp}
\email{tachiya@hirosaki-u.ac.jp}

\makeatletter
\@namedef{subjclassname@2020}{%
  \textup{2020} Mathematics Subject Classification}
\makeatother
\keywords{square-full integers, $k$-full integers, asymptotic density, uniform distribution}

\subjclass[2020]{11N25, 11N64, 11J71}

\date{\today}

%\dedicatory{Dedicated to Professor~Iekata~Shiokawa on the~occasion of his~85th~birthday.}

\begin{abstract} 
Let $k\geq2$ be an integer. 
The aim of this paper is to investigate the distribution 
of $k$-full integers between three successive 
$k$-th powers. 
More precisely, for any integers $\ell,m\ge0$, we establish the explicit asymptotic density 
for the set of integers $n$ such that the intervals 
$(n^k, (n+1)^k)$ and $((n+1)^k, (n+2)^k)$ contain exactly $\ell$ and $m$ $k$-full integers, respectively. 
As an application, we prove that there are  
infinitely many triples of successive $k$-th powers in the sequence of $k$-full integers,  
thereby providing a more general answer to Shiu's question. 
\end{abstract}

\maketitle

%%%%%%%%%%%%%%%%%%%%%%%%%%%%%%%%%%%%%%%%%%%%%%%%%%%%%%%%%%%%%%%%%%%%%%%%%
\section{Introduction and main results}\label{sec1}

Throughout this paper, let 
$k\geq2$ be an integer. 
A positive integer $n$ is called a $k$-full integer 
if $p^k$ divides $n$ for every prime factor $p$ of $n$. 
When $k=2$, such integers are known as square-full or powerful integers. 
Let $\mathcal{F}_k$ be 
the set of all $k$-full integers. 

For a set $\mathcal{A}$ of positive integers and a real number $x>1$, 
let $\mathcal{A}(x)$ 
denote the set of integers in $\mathcal{A}$ not exceeding $x$. 
Moreover, let $\#\mathcal{A}$ denote the number of elements in the finite set $\mathcal{A}$. 
In 1934, Erd{\H o}s and Szekeres \cite{ES34} established the asymptotic formula 
\[
\#{\mathcal F}_k(x)=c_kx^{1/k}+O\bigl(x^{1/(k+1)}\bigr)\quad (x\to\infty) 
\]
with an explicit positive constant $c_k$; in particular,  
\[
\#{\mathcal F}_2(x)=c_2x^{1/2}+O\bigl(x^{1/3}\bigr)\quad (x\to\infty),
\]
where $c_2:={\zeta(3/2)}/{\zeta(3)}=2.173\dots$ and $\zeta(s)$ is the Riemann zeta function. 
Bateman and Grosswald~\cite{BG58} improved the $O$-estimate of the error term (see also \cite[\S 14.4 and p.~438--439]{I03}). 
%%%%%%%%%%%%%%%%%%%%%%%%%%%%%%%%%%
In related work, Shiu~\cite{S80} investigated the distribution of square-full integers between successive squares. 
Let $\ell\ge0$ be an integer and 
\begin{equation}\label{eq:102533}
\mathcal{A}_{\ell}^{(k)}:=
\{
n\in\mathbb{Z}_{\ge1}
\mid
\text{$\#( (n^k,(n+1)^k)\cap \mathcal{S}_{k})=\ell$}
\},
\end{equation}
where $\mathcal{S}_{k}$ denotes the set of all $k$-full integers 
that are not perfect $k$-th powers, and thus, 
$\mathcal{A}_{\ell}^{(k)}$ defines the set of positive integers $n$ 
for which the interval $(n^k,(n+1)^k)$ 
contains exactly $\ell$ $k$-full integers in $\mathcal{S}_k$.  
%%%%%%%%%%%%%%%%%%%%%%%%%%%%%%%%%%%%%%%%%%%%%%%%%%%%%%%%%%%%%
Shiu proved in \cite{S80} that, for each integer $\ell\ge0$, there exists 
an explicit positive constant $d_\ell$ such that 
\begin{equation}\label{eq:Shi0203}
\#\mathcal{A}_{\ell}^{(2)}(x)=d_{\ell} x+o(x) \quad (x\to\infty).
\end{equation}
%%%%%%%%%%%%%%%%%%%%%%%%%%
The first few numerical values of $d_{\ell}$ are given in \cite[p.~176]{S80}; e.g., $d_{0}=0.275\dots$, $d_{1}=0.395\dots$, $d_{2}=0.231\dots$. 
In the case $\ell=0$, 
De Koninck and Luca~\cite{KF04} provided a more precise asymptotic formula: 
%with an error term: 
\[
\#\mathcal{A}_{0}^{(2)}(x)=d_{0}x+O(x/(\log\log x)^{\frac{1}{2}})\quad (x\to\infty).
\]
%%%%%%%%%%%%%%%%%%%%%%%%%%%%%
Furthermore, Shiu's result was 
extended to $k$-full integers 
by 
Xiong and Zaharescu~\cite{XZ11}, who established the asymptotic formula
\begin{equation}\label{eq:XZ0203}
\#\mathcal{A}_{\ell}^{(k)}(x)=d_{\ell}^{(k)}x+O(x/(\log\log x)^{\frac{1}{2k}})\quad 
(x\to\infty),
\end{equation}
with explicit positive constant $d_{\ell}^{(k)}$ for each $k\geq2$ and $\ell\ge0$. 
%,such as $(1,4)$, $(9,16)$, $(16,25)$, and so on. 
They also showed that, for each integer $k\ge2$, 
the generating function of $(d_{\ell}^{(k)})_{\ell\ge0}$ is given by 
\begin{equation}\label{gene}
\sum_{\ell\ge 0}^{}d_{\ell}^{(k)}z^\ell=\prod_{\lambda\in\Lambda_k}^{}
\left(1+\frac{z-1}{\lambda}\right),
\end{equation}
where the product is taken over all real numbers $\lambda>2$ in the set 
\begin{equation}\label{Lambda_k}
\Lambda_k:=
\left.\left\{
(b_1^{k+1}\cdots b_{k-1}^{2k-1})^{1/k}
\,\right|\, 
b_1,\dots,b_{k-1}\in\mathbb{Z}_{\geq1},\,\, b_1\cdots b_{k-1}\ge2,\,\,\mu^2(b_1\cdots b_{k-1})=1
\right\}
\end{equation}
and $\mu$ is the M{\" o}bius function. 
Expanding the right-hand side of \eqref{gene} and comparing the coefficients yields
\begin{equation}\label{eq:1124}
d_{\ell}^{(k)}
=
\sum_{n\ge0}^{}(-1)^n
\binom{\ell+n}{\ell}
\xi_{\ell+n}^{(k)},
\end{equation}
where $(\xi_{r}^{(k)})_{r\ge0}$ is a sequence defined by 
\begin{equation}\label{eq:112514}
\xi_{0}^{(k)}:=1,\qquad 
\xi_{r}^{(k)}:=
\sum_{\substack{\mathcal{L}\subseteq \Lambda_k\\\#\mathcal{L}=r}}\prod_{\lambda\in\mathcal{L}}\frac{1}{\lambda},\quad r=1,2,\dots.
\end{equation}
Shiu~\cite{S80} previously 
obtained the expression \eqref{eq:1124} for the case $k=2$.

The aim of this paper is to investigate the distribution of $k$-full integers 
in the wider interval $(n^k,(n+2)^k)$. 
More precisely, in Theorem~\ref{thm1} below, we establish the explicit asymptotic densities for the sets
\[
\mathcal{A}_{\ell,m}^{(k)}:=
\left\{
n\in\mathbb{Z}_{\ge1}\,
\middle|\,
\substack{
\text{$\#\big( (n^k,(n+1)^k)\cap \mathcal{S}_{k}\bigr)=\ell$}, \\
\text{$\#\big( ((n+1)^k,(n+2)^k)\cap \mathcal{S}_{k}\bigr)=m$}
}
\right\}
\]
for all $\ell,m\ge0$, 
and show that these densities are all positive. 
In particular, the sets $\mathcal{A}_{\ell,m}^{(k)}$ are 
infinite for all non-negative integers $\ell$ and $m$. 
This provides a more general answer to Shiu's question \cite[p.~172, lines 10--13]{S80} regarding the distribution of squares in the sequence of square-full integers. 
In Section~\ref{sec6}, we give explicit expressions for the asymptotic densities of 
$\mathcal{A}_{\ell,m}^{(k)}$ $(\ell,m\ge0)$ 
and, as an application, recover the result~\eqref{gene} of Xiong and Zaharescu. \\
%the classical Weyl criterion 

Before stating our results, we need some notation. 
It is known that any $k$-full integer $n$ has the unique representation  
\begin{equation}\label{eq:0}
n=a^k b_1^{k+1}\cdots b_{k-1}^{2k-1},
\end{equation}
where $a,b_1,\dots,b_{k-1}$ are positive integers such that $b_1\cdots b_{k-1}$ is a square-free integer.   
For instance, a square-full integer $n$ is uniquely written as 
$n=a^2b^3$ with a positive integer $a$ and a square-free integer $b$. 
From definition~\eqref{Lambda_k} of 
$\Lambda_k$ and expression \eqref{eq:0}, every $k$-full integer $n$ can be uniquely 
represented as $n=a^k\lambda^k$ for an integer $a\ge1$ and a real number $\lambda\in\Lambda_k\cup\{1\}$.  
For a non-empty subset $\mathcal{I}\subseteq \Lambda_k$, 
we define 
\[
\mathcal{S}_\mathcal{I}:=\{a^k\lambda^k\mid a\in\mathbb{Z}_{\ge1},\,\lambda\in \mathcal{I}\}\subseteq \mathcal{S}_k.
\]
In particular, $\mathcal{S}_{\Lambda_k}(=\mathcal{S}_{k})$ is the set of all $k$-full integers that are not perfect $k$-th powers. 
For the empty set $\varnothing\subseteq \Lambda_k$, we define   
$\mathcal{S}_{\varnothing}:=\varnothing$ and $\#\mathcal{S}_{\varnothing}:=0$.  
%%%%%%%%%%%%%%%%%%%%%%%%%%%%%%%%%%%%%%%%%
\begin{definition}\label{def2}
Let $\mathcal{A}$ be a set of positive integers. 
The asymptotic density of $\mathcal{A}$ is defined by 
\begin{equation}\label{eq:1103}
d(\mathcal{A}):=\lim_{x\to\infty}\frac{\#\mathcal{A}(x)}{x},
\end{equation}
provided the limit exists. In this case, 
we say that $\mathcal{A}$ has asymptotic density $d(\mathcal{A})$.
\end{definition}
Clearly, if the set $\mathcal{A}$ has asymptotic density $d(\mathcal{A})$, then 
definition~\eqref{eq:1103} is equivalent to the asymptotic formula 
$\#\mathcal{A}(x)=d(\mathcal{A})x+o(x)$ $(x\to\infty)$.  
Our results are the following. 
%%%%%%%%%%%%%%%%%%%%%%%%%%%%%%%%%%%%%%%%%%%%%%
\begin{theorem}\label{thm2}
Let $\mathcal{I}$ and $\mathcal{J}$ be finite subsets of $\Lambda_k$ with $\mathcal{I}\cap \mathcal{J}=\varnothing$. 
Then the set 
\begin{equation}\label{AIJ}
\mathcal{B}_{\mathcal{I,J}}^{(k)}:=
\left\{
n\in\mathbb{Z}_{\ge1}
\,\left|\, 
\substack{
\text{$\#\big( (n^k,(n+1)^k)\cap \mathcal{S}_\mathcal{I}\bigr)=\# \mathcal{I}$}, \\
\text{$\#\big( ((n+1)^k,(n+2)^k)\cap \mathcal{S}_\mathcal{J}\bigr)=\# \mathcal{J}$},\\ 
\text{$( n^k,(n+2)^k)\cap \mathcal{S}_{\Lambda_k\setminus (\mathcal{I}\cup \mathcal{J})}=\varnothing$}
}
\right\}
\right.
\end{equation}
has positive asymptotic density 
\begin{equation}\label{eq:1013}
d(\mathcal{B}_{\mathcal{I,J}}^{(k)})=
\prod_{\lambda\in \mathcal{I}\cup \mathcal{J}}
\frac{1}\lambda{}\cdot \prod_{\lambda\in \Lambda_k\setminus(\mathcal{I}\cup\mathcal{J})}\bigg(1-\frac{2}{\lambda}\bigg).
\end{equation}
\end{theorem}
%%%%%%%%%%%%%%%%%%%%%%%%%%%%%%%%%%%%%%%%%%%%%%%%%%%%%%%%%%%%%%%%%%%%%%%%
Note that the infinite product in the right-hand side of \eqref{eq:1013} 
converges since 
\begin{equation}\label{conv}
\sum_{\lambda\in \Lambda_k}^{}\frac{1}{\lambda}\le 
\sum_{b_1,\dots,b_{k-1}\ge1}\frac{1}{(b_1^{k+1}\cdots b_{k-1}^{2k-1})^{1/k}}=\prod_{j=1}^{k-1}\zeta
\bigg(
1+\frac{j}{k}
\bigg)<\infty.
\end{equation}
%%%%%%%%%%%%%%%%%%%%%%%%%%%%%%%%%%%%%%%%%%%%%%%%%%%%%%%%%%%%%%%%%%%%%%%%
The expression \eqref{eq:1013} shows that the asymptotic density $d(\mathcal{B}_{\mathcal{I,J}}^{(k)})$
depends only on the union $\mathcal{I}\cup\mathcal{J}$, 
rather than on the individual choices of $\mathcal{I}$ and $\mathcal{J}$; 
in particular, we have %the symmetric relation 
\begin{equation}\label{symme}
d(\mathcal{B}_{\mathcal{I,J}}^{(k)})=d(\mathcal{B}_{\mathcal{J,I}}^{(k)})
\end{equation}
for any pair of finite subsets $\mathcal{I},\mathcal{J}\subseteq\Lambda_k$ with $\mathcal{I}\cap \mathcal{J}=\varnothing$. 
When $\mathcal{I}=\mathcal{J}=\varnothing$, 
Theorem~\ref{thm2} reduces to the following Corollary~\ref{cor12}. 
For an integer $k\ge2$, let $C_k$ be the positive constant defined by  
\[
C_{k}:=\prod_{\lambda\in\Lambda_k}^{}\left(
1-\frac{2}{\lambda}
\right).
\] 
%%%%%%%%%%%%%%%%%%%%%%%%%%%%%%%%%%%%%%%%%%%%%%%%%%%%%%%%%%%%%%%%%%%%%%%%%%%%%%
\begin{corollary}\label{cor12}
The set  
\[
\mathcal{B}_{\varnothing,\varnothing}^{(k)}=
\{
n\in\mathbb{Z}_{\ge1}\mid ( n^k,(n+2)^k)\cap \mathcal{S}_{k}=\varnothing
\}
\]
has positive asymptotic density $d(\mathcal{B}_{\varnothing,\varnothing}^{(k)})=C_{k}$.
In particular, there are infinitely many integers $n$ such that the interval $(n^k,(n+2)^k)$ contains no $k$-full integers except for $(n+1)^k$.
\end{corollary}
%%%%%%%%%%%%%%%%%%%%%%%%%%%%%%%%%%%%%%%%%%%%%%%%%%%%%%%%%%%%%%%%%%%%%%%%%
For example, the case $k=2$ of Corollary~\ref{cor12} asserts 
that the set
\[
\mathcal{B}_{\varnothing,\varnothing}^{(2)}=
\left\{
n\in\mathbb{Z}_{\ge1}\,\left|\, \substack{\text{the interval $( n^2,(n+2)^2)$ contains }\\ \text{no square-full integers other than $(n+1)^2$}}
\right
\}\right.=\{3,6,12,23,26,34,\dots\}
\]
has positive asymptotic density  
\[
d(\mathcal{B}_{\varnothing,\varnothing}^{(2)})=C_{2}=\prod_{n\ge2}^{}\bigg(1-\frac{2\mu^2(n)}{n^{3/2}}\bigg)=0.049227\dots.%\qquad \text{and}  \qquad 
\]
%%%%%%%%%%%%%%%%%%%%%%%%%%%%%%%%%%%%%%%%%%%%%%%%%%%%%%%%%%%%%%%%%%%%%%%%%%%%%%
\begin{remark}\label{rem0203}
Shiu's formula~\eqref{eq:Shi0203} with $\ell=0$ implies that the set $\mathcal{A}_{0}^{(2)}$ is infinite; that is,   
there exist infinitely many pairs of 
consecutive perfect squares in the sequence of square-full integers, 
such as 
$(1,4),(9,16)$, and $(16,25)$ 
(Note that the pair $(4, 9)$ is excluded because the square-full integer $8$ lies between them). 
A similar property also holds for $k$-full integers for any $k\ge3$ by \eqref{eq:XZ0203}. % \eqref{eq3}. 
Corollary~\ref{cor12} provides a stronger result, establishing the existence of 
infinitely many triples of consecutive perfect $k$-th powers in the sequence of 
$k$-full integers; for instance, when $k=2$, such triples include
$(9,16,25),(36,49,64)$, and $(144,169,196)$. 
This is the best possible result in this direction, as there are no quadruples of consecutive $k$-th powers in the sequence of $k$-full integers. 
Indeed, for any integer $n\ge1$, the inequality $2^{1+1/k}\le 2\sqrt{2}<3$ 
guarantees that there exists an integer $a\geq1$ satisfying  
\[n^k<a^k2^{k+1}<(n+3)^k,\]
where $a^k2^{k+1}$ is a $k$-full integer but not a $k$-th power. 
\end{remark}
%%%%%%%%%%%%%%%%%%%%%%%%%%%%%%%%%%%%%%%%%%%%%%%%%%%%%%%%%%%%%%%%%%%%%%%%%%%%%%%
\begin{theorem}\label{thm1}
Let $\ell,m\ge0$ be integers. Then the set  
\begin{equation}\label{eq:1024}
\mathcal{A}_{\ell,m}^{(k)}:=
\left\{
n\in\mathbb{Z}_{\ge1}
\left|
\substack{
\text{$\#\big( (n^k,(n+1)^k)\cap \mathcal{S}_{k}\bigr)=\ell$}, \\
\text{$\#\big( ((n+1)^k,(n+2)^k)\cap \mathcal{S}_{k}\bigr)=m$}
}
\right\}\right.
\end{equation}
has positive asymptotic density
\begin{equation}\label{eq:10121}
d(\mathcal{A}_{\ell,m}^{(k)})=
\sum_{\substack{\mathcal{I},\mathcal{J}\subseteq \Lambda_k\\\#\mathcal{I}=\ell,\, \#\mathcal{J}=m,\, \mathcal{I}\cap \mathcal{J}=\varnothing}}d(\mathcal{B}_{\mathcal{I,J}}^{(k)}),
\end{equation}
where $d(\mathcal{B}_{\mathcal{I,J}}^{(k)})$ is 
the asymptotic density of $\mathcal{B}_{\mathcal{I,J}}^{(k)}$ given in \eqref{eq:1013}.  
\end{theorem}
%%%%%%%%%%%%%%%%%%%%%%%%%%%%%%%%%%%%%%%%%%%%%%%

%%%%%%%%%%%%%%%%%%%%%%%%%%%%%%%%%%%%%%%%%%%%
From \eqref{symme} and \eqref{eq:10121}, we have 
the symmetry 
$d(\mathcal{A}_{\ell,m}^{(k)})=d(\mathcal{A}_{m,\ell}^{(k)})$ for any non-negative integers $\ell$ and $m$. 
%%%%%%%%%%%%%%%%%%%%%%%%%%%%%%
Moreover, by \eqref{eq:1013} and \eqref{eq:10121}, we have  
\begin{equation}\label{eq:1205}
d(\mathcal{A}_{\ell,m}^{(k)})=
\sum_{\substack{\mathcal{I},\mathcal{J}\subseteq \Lambda_k\\\#\mathcal{I}=\ell,\, \#\mathcal{J}=m,\, \mathcal{I}\cap \mathcal{J}=\varnothing}}
\prod_{\lambda\in \mathcal{I}\cup \mathcal{J}}
\frac{1}\lambda{}\cdot \prod_{\lambda\in \Lambda_k\setminus(\mathcal{I}\cup\mathcal{J})}\bigg(1-\frac{2}{\lambda}\bigg),
\end{equation}
which immediately yields the generating function of the sequence $(\mathcal{A}_{\ell,m}^{(k)})_{\ell,m\ge0}$: 
\[
\sum_{\ell,m\ge0}^{}d(\mathcal{A}_{\ell,m}^{(k)})z^\ell w^m=
\prod_{\lambda\in\Lambda_k}
\left(
1+\frac{z+w-2}{\lambda}
\right).
\]
The first few numerical values of $d(\mathcal{A}_{\ell,m}^{(k)})$ for $k=2,3$, computed using the formula~\eqref{eq:20251125}, 
are provided in Tables~\ref{table1} and \ref{table2} in Section~\ref{sec6}; for example, when $k=2$, we have 
\begin{gather*}
d(\mathcal{A}_{0,0}^{(2)})=d(\mathcal{B}_{\varnothing,\varnothing}^{(2)})= 0.049\dots,\qquad 
d(\mathcal{A}_{0,1}^{(2)})=d(\mathcal{A}_{1,0}^{(2)})= 0.107\dots,\\
d(\mathcal{A}_{1,1}^{(2)})=0.158\dots,\qquad 
d(\mathcal{A}_{0,2}^{(2)})=d(\mathcal{A}_{2,0}^{(2)})= 0.079\dots,\qquad 
d(\mathcal{A}_{0,3}^{(2)})=d(\mathcal{A}_{3,0}^{(2)})= 0.030\dots.
\end{gather*}
%%%%%%%%%%%%%%%%%%%%%%%%%%%%%%%%%%%%%%%%%%%%%%%%%%
%%%%%%%%%%%%%%%%%%%%%%%%%%%%%%%%%%%%%%%%%%%%%%%%%%%%%%%%%%%%
\begin{corollary}\label{cor:1025}
Let $\ell\ge0$ be an integer, and let $\mathcal{A}_{\ell}^{(k)}$ be defined as in \eqref{eq:102533}. 
Then we have  
\begin{equation}\label{eq:1025444}
d(\mathcal{A}_{\ell}^{(k)})=
\sum_{m\ge0}^{}d(\mathcal{A}_{\ell,m}^{(k)})=
\sum_{m\ge0}^{}d(\mathcal{A}_{m,\ell}^{(k)}).
\end{equation}
\end{corollary}
%%%%%%%%%%%%%%%%%%%%%%%%%%%
Corollary~\ref{cor:1025} states that 
the asymptotic density is countably additive with respect to 
the disjoint unions
$
\mathcal{A}_{\ell}^{(k)}=\sqcup_{m\ge0}
\mathcal{A}_{\ell,m}^{(k)}=
\sqcup_{m\ge0}\mathcal{A}_{m,\ell}^{(k)}
$, which is non-trivial since the asymptotic density is 
not countably additive in general. \\
%%%%%%%%%%%%%%%%%%%%%%%%%%%%%%%%%
%%%%%%%%%%%%%%%%%%

This paper is organized as follows. 
In Section~\ref{sec2}, we show Lemma~\ref{lem:1016},  
which plays a crucial role in the proof of Theorem \ref{thm2}. 
Sections~\ref{sec3} and \ref{sec4} are dedicated to the proofs of 
Theorems~\ref{thm2} and \ref{thm1}, respectively; we note that 
Theorem~\ref{thm1} is derived from Theorem~\ref{thm2}. 
In Section \ref{sec5}, we prove Corollary~\ref{cor:1025} 
as an application of Theorem~\ref{thm1}. 
It should be noted that our proofs rely on 
the classical multidimensional equidistribution theorem~(see Lemma~\ref{lem1} in Section~\ref{sec2}); 
in particular, we do not require discrepancy estimates such as the Koksma--Hlawka or  Erd{\H o}s--Tur{\'a}n--Koksma inequalities 
used in \cite{KF04,XZ11}, as our investigation focuses on the asymptotic density. 
Finally, Section~\ref{sec6} provides explicit formulas 
for $d(\mathcal{A}_{\ell,m}^{(k)})$ and discusses the maximum values for $k=2$ and $3$. 
%the maximum values of these densities for 
%and present numerical examples for the cases $k=2$ and $k=3$. 
%velues. 
%Our proofs are different from those used in previous works \cite{KF04,S80,XZ11}. 

%%%%%%%%%%%%%%%%%%%%%

\section{Preparation for proof of Theorem~\ref{thm2}}\label{sec2}
We first prepare the following Lemmas \ref{lem1}--\ref{lem:1018} for the proof of Lemma~\ref{lem:1016}. 
Let $\{x\}$ denote the fractional part of a real number $x$.

%%%%%%%%%%%%%%%%%%%%%%%%%%%%%%%%%%%%%%%%%
\begin{lemma}[{cf. \cite[p.~48, Example~6.1]{KN74}}]\label{lem1}
If the real numbers $1,\alpha_1,\dots,\alpha_m$ are linearly independent over $\mathbb{Q}$, 
then the sequence of $m$-tuples $\boldsymbol{\alpha}_n:=
\bigl(\{\alpha_1n\},\dots,\{\alpha_mn\}\bigr)$ $(n\ge1)$ is uniformly distributed modulo~$1$. That is,  
\[
\lim_{x\to\infty}\frac{\#\{n\in\mathbb{Z}_{\ge1}\mid \boldsymbol{\alpha}_n\in[\boldsymbol{a},\boldsymbol{b}), n\le x\}}{x}=\prod_{j=1}^{m}(b_j-a_j)
\]
for any interval $[\boldsymbol{a},\boldsymbol{b}):=\prod_{j=1}^{m}[a_j,b_j)\subseteq [0,1)^m$. 
\end{lemma}
%%%%%%%%%%%%%%%%%%%%%%%%%%%%%%%%%%%%%%%
Let $\Lambda_k$ be the set of real numbers $>2$ defined by \eqref{Lambda_k}. 
%%%%%%%%%%%%%%%%%%%%%%%%%%%%%%%%%%%%%%
\begin{lemma}\label{linind}
If the numbers $\lambda_1,\dots,\lambda_n\in\Lambda_k$ are distinct, 
then the numbers $1,\lambda_1^{-1},\dots,\lambda_n^{-1}$ are linearly independent over $\mathbb{Q}$. 
\end{lemma}
%%%%%%%%%%%%%%%%%%%%%%%%%%%%%%%%%%%%%%%
\begin{proof}
Let $\lambda_1,\dots,\lambda_n\in\Lambda_k$ be distinct numbers. 
By \eqref{Lambda_k}, each 
$\lambda_i$ can be written as  
\begin{equation}\label{form}
\lambda_i=\prod_{j=1}^{k-1}\prod_{p\in \mathcal{P}_{i,j}\cup\{1\}}^{}p^{1+\frac{j}{k}},
\end{equation}
where $\mathcal{P}_{i,j}$ $(j=1,\dots,k-1)$ are finite subsets of 
prime numbers, not all empty, with $\mathcal{P}_{i,j_1}\cap\mathcal{P}_{i,j_2}=\varnothing$ $(j_1\neq j_2)$. 
Let $\cup_{i,j}\mathcal{P}_{i,j}=\{p_1,\dots,p_m\}$. 
Suppose to the contrary that the numbers $1,\lambda_1^{-1},\dots,\lambda_n^{-1}$ are linearly dependent over $\mathbb{Q}$. 
Then, by \eqref{form}, there exists a non-zero polynomial $Q(X_1,\dots,X_m)\in\mathbb{Z}[X_1,\dots,X_m]$, 
with the degree in each variable $X_j$ less than $k$, such that 
% with integer coefficients of degree less than $k$ with respect to each 
%$X_1,\dots,X_m$ 
%for which 
$Q(p_1^{1/k},\dots,p_m^{1/k})=0$. 
This contradicts \cite[Theorem~2]{Be39}, and thus, Lemma~\ref{linind} is proved.  
\end{proof}
%%%%%%%%%%%%%%%%%%%%%%%%%%%%%%%%%%%%%%%%
\begin{lemma}\label{lem:1018}
Let  $n\ge1$ be an integer and $\lambda\in \Lambda_k$. 
Then, for each integer $j=1,2$, the following properties are equivalent:
\begin{enumerate}[{\rm (i)}]
\item\label{s1} 
$(n^k,(n+j)^k)\cap \mathcal{S}_{\{\lambda\}}\neq\varnothing$.

\item\label{s3} 
$\#\bigl((n^k,(n+j)^k)\cap \mathcal{S}_{\{\lambda\}}\bigr)=1$.

\item\label{s2}
$\left\{\frac{n}{\lambda}\right\}>1-\frac{j}{\lambda}$.
\end{enumerate}
\end{lemma}
%%%%%%%%%%%%%%%%%%%%%%%%%%%%%%%%%%%%%%%%%%%
\begin{proof}
We first assume \eqref{s1}. 
Let $a_1^k\lambda^k,a_2^k\lambda^k\in (n^k,(n+j)^k)\cap\mathcal{S}_{\{\lambda\}}$. 
Then 
$n< a_1\lambda, a_2\lambda<n+j$ and so $\lambda|a_1-a_2|<j\le 2$. 
Since $\lambda>2$ and $a_1,a_2$ are integers, we obtain $a_1=a_2$, which shows \eqref{s3}.  
Next we prove \eqref{s3}$\Rightarrow$\eqref{s2}. 
If there exists an integer $a\ge1$ satisfying $n^k<a^k\lambda^k<(n+j)^k$, we have $a-1<n/\lambda<a<(n+j)/\lambda$ since $0<j/\lambda<1$ and $a$ is an integer, and so  
\[
\left\{\frac{n}{\lambda}\right\}=\frac{n}{\lambda}-(a-1)>\frac{n}{\lambda}+1-\frac{n+j}{\lambda}=1-\frac{j}{\lambda}.
\]
Finally, we assume \eqref{s2}. 
Then 
\[
\frac{n}{\lambda}-\left\lfloor \frac{n}{\lambda}\right\rfloor=\left\{\frac{n}{\lambda}\right\}>1-\frac{j}{\lambda},\]
where $\lfloor n/\lambda\rfloor$ denotes the integer part of $n/\lambda$, 
so that 
\[
\frac{n}{\lambda}<a:=1+\left\lfloor \frac{n}{\lambda}\right\rfloor <\frac{n+j}{\lambda}.
\]
Hence, we have $a^k\lambda^k\in(n^k,(n+j)^k)\cap \mathcal{S}_{\{\lambda\}}$ 
and property \eqref{s1} holds. Lemma~\ref{lem:1018} is proved. 
\end{proof}
%%%%%%%%%%%%%%%%%%%%%%%%%%%%%%%%%%%%%%%%%%
%%%%%%%%%%%%%%%%%%%%%%%%%%%%%%%%%%%%%%%%%%%

\begin{lemma}\label{lem:1016}
Let $\mathcal{I}_1,\mathcal{I}_2,\mathcal{I}_3$ be finite subsets of $\Lambda_k$ with 
$\mathcal{I}_i\cap\mathcal{I}_j=\varnothing$ $(i\neq j)$.  
Then the set  
\begin{equation}\label{setA}
\mathcal{B}:=
\left\{
n\in\mathbb{Z}_{\ge1}\left| \substack{\text{$\#\big( (n^k,(n+1)^k)\cap \mathcal{S}_{\mathcal{I}_1}\bigr)=\# \mathcal{I}_1$}, \\[0.05cm]
\text{$\#\big( ((n+1)^k,(n+2)^k)\cap \mathcal{S}_{\mathcal{I}_2}\bigr)=\# \mathcal{I}_2$,}\\[0.05cm]
\text{$( n^k,(n+2)^k)\cap \mathcal{S}_{\mathcal{I}_3}=\varnothing$}}
\right\}\right.
\end{equation}
has positive asymptotic density 
\[
d(\mathcal{B})=\displaystyle\prod_{\lambda\in \mathcal{I}_1\cup \mathcal{I}_2}
\frac{1}\lambda{}\cdot \prod_{\lambda\in \mathcal{I}_3}\left(1-\frac{2}{\lambda}\right).
\]
\end{lemma}
%%%%%%%%%%%%%%%%%%%%%%
\begin{proof}
Let $\mathcal{I}\subseteq\Lambda_k$ be a finite set. 
Then we have $\mathcal{S}_{\mathcal{I}}=\sqcup_{\lambda\in\mathcal{I}} \mathcal{S}_{\{\lambda\}}$ (disjoint union) since every $k$-full integer is represented uniquely as \eqref{eq:0}. 
Hence, for any integer $n\ge1$ and each integer $j=1,2$, 
we obtain by Lemma~\ref{lem:1018}
\begin{equation}\label{d1}
\#\bigl(
(n^k,(n+j)^k)\cap \mathcal{S}_\mathcal{I}
\bigr)
=\sum_{\lambda\in \mathcal{I}}\#
\bigl(
(n^k,(n+j)^k)\cap \mathcal{S}_{\{\lambda\}}
\bigr)
=
\#\left\{
\lambda\in \mathcal{I}\,\left|\, \left\{\tfrac{n}{\lambda}\right\}\in\bigl(1-\tfrac{j}{\lambda},1\bigr)
\right\}\right.
\end{equation}
and 
\begin{align}
\#\bigl(
((n+1)^k,(n+2)^k)\cap \mathcal{S}_\mathcal{I}
\bigr)&=
\#\bigl(
(n^k,(n+2)^k)\cap \mathcal{S}_\mathcal{I}
\bigr)-
\#\bigl(
(n^k,(n+1)^k)\cap \mathcal{S}_\mathcal{I}
\bigr)\nonumber\\
&=
\#\left\{
\lambda\in \mathcal{I}\,\left|\, \left\{\tfrac{n}{\lambda}\right\}\in\bigl(1-\tfrac{2}{\lambda},1-\tfrac{1}{\lambda}\bigr)
\right\},\right.\label{d2}
\end{align}
where we note that $\{\tfrac{n}{\lambda}\}\neq 0,1-\tfrac{2}{\lambda},1-\tfrac{1}{\lambda}$ since $\lambda\in\mathcal{I}$ is irrational. 
Thus, by \eqref{d1} and \eqref{d2}, the set $\mathcal{B}$ in \eqref{setA} is given by  
\[
\mathcal{B}=
\left\{
n\in\mathbb{Z}_{\ge1}\,\left|\, \substack{\text{$\left\{\tfrac{n}{\lambda}\right\}\in\Bigl(1-\tfrac{1}{\lambda},1\Bigr)$ for any $\lambda\in\mathcal{I}_1$}, \\
\text{$\left\{\tfrac{n}{\lambda}\right\}\in\Bigl(1-\tfrac{2}{\lambda},1-\tfrac{1}{\lambda}\Bigr)$ for any $\lambda\in\mathcal{I}_2$,}\\
\text{$\left\{\tfrac{n}{\lambda}\right\}\in\Bigl(0,1-\tfrac{2}{\lambda}\Bigr)$ for any $\lambda\in\mathcal{I}_3$}}
\right\}.\right.
\]
Therefore, Lemma \ref{lem:1016} follows from Lemmas \ref{lem1} and \ref{linind}. 
\end{proof}

%%%%%%%%%%%%%%%%%%%%%%%%%%%
%%%%%%%%%%%%%%%%%%%%%%%%%%%%%%
%%%%%%%%%%%%%%%%%%%%%%%%%%%%%
%%%%%%%%%%%%%%%%%%%%%%%%%%%%%%%
%%%%%%%%%%%%%%%%%%%%%%%%%
\section{Proof of Theorem \ref{thm2}}\label{sec3}

Let $\mathcal{I},\mathcal{J}\subseteq \Lambda_k$ be as in Theorem~\ref{thm2} 
and 
\begin{equation}\label{Lambdaset}
\Lambda_k=\{\lambda_j\mid 2<\lambda_1<\lambda_2<\cdots\}.
\end{equation}
Let $\varepsilon>0$ be an arbitrary constant. 
By \eqref{conv} there exists the least positive integer $N=N(\varepsilon)$ 
such that the set 
$\mathcal{L}:=\{\lambda_1,\lambda_2,\dots,\lambda_N\}\subseteq\Lambda_k$ satisfies the following properties: 
\begin{enumerate}[(i)]

%\item
%$\mathcal{L}:=\{\lambda_1,\lambda_2,\dots,\lambda_N\}\supseteq \mathcal{I}\cup \mathcal{J}$.

\item\label{prop2}
$\displaystyle\sum_{\lambda\in\Lambda_k\setminus \mathcal{L}}^{}
{\lambda}^{-1}%=\sum_{j>N}^{}{\lambda_j}^{-1}
<\varepsilon$.

\item\label{pro3}
$
0<d_\mathcal{I,J,L}^{(k)}-d_\mathcal{I,J}^{(k)}<
\varepsilon$,
where
\[
d_\mathcal{I,J}^{(k)}:=
\displaystyle\prod_{\lambda\in \mathcal{I}\cup \mathcal{J}}
\frac{1}\lambda{}\cdot \prod_{\lambda\in \Lambda_k\setminus(\mathcal{I}\cup\mathcal{J})}\left(1-\frac{2}{\lambda}\right),\qquad 
d_\mathcal{I,J,L}^{(k)}:=
\displaystyle\prod_{\lambda\in \mathcal{I}\cup \mathcal{J}}
\frac{1}\lambda{}\cdot \prod_{\lambda\in \mathcal{L}\setminus(\mathcal{I}\cup\mathcal{J})}\left(1-\frac{2}{\lambda}\right).%=
\]
\end{enumerate}
Let $\mathcal{B}_\mathcal{I,J}^{(k)}$ be as in \eqref{AIJ} and define 
%$\mathcal{A}_\mathcal{I,J,K}$ be a subset of positive integers defined by 
\[
\mathcal{B}_\mathcal{I,J,L}^{(k)}:=
\left\{
n\in\mathbb{Z}_{\ge1}\left| \substack{\text{$\#\big( (n^k,(n+1)^k)\cap \mathcal{S}_{\mathcal{I}}\bigr)=\# \mathcal{I}$}, \\[0.05cm]
\text{$\#\big( ((n+1)^k,(n+2)^k)\cap \mathcal{S}_{\mathcal{J}}\bigr)=\# \mathcal{J}$,}\\[0.05cm]
\text{$( n^k,(n+2)^k)\cap \mathcal{S}_{\mathcal{L}\setminus (\mathcal{I}\cup\mathcal{J})}=\varnothing$}}
\right\}.\right.
\] 
Then, applying Lemma~\ref{lem:1016} with 
\[
\mathcal{B}:=\mathcal{B}_\mathcal{I,J,L}^{(k)},\quad 
\mathcal{I}_1:=\mathcal{I}, \quad 
\mathcal{I}_2:=\mathcal{J}, \quad 
\mathcal{I}_3:=\mathcal{L}\setminus(\mathcal{I}\cup\mathcal{J})
\]
and using property \eqref{pro3}, we obtain   
\begin{align}
\Bigg|
\frac{\#\mathcal{B}_\mathcal{I,J}^{(k)}(x)}{x}-d_\mathcal{I,J}^{(k)}
\Bigg|&\leq 
\Bigg|
\frac{\#\mathcal{B}_\mathcal{I,J}^{(k)}(x)}{x}-\frac{\#\mathcal{B}_\mathcal{I,J,L}^{(k)}(x)}{x}
\Bigg|+
\Bigg|
\frac{\#\mathcal{B}_\mathcal{I,J,L}^{(k)}(x)}{x}-d_\mathcal{I,J,L}^{(k)}
\Bigg|
+\big|d_\mathcal{I,J,L}^{(k)}-d_\mathcal{I,J}^{(k)}\big|,\nonumber\\
&<\Bigg|
\frac{\#\mathcal{B}_\mathcal{I,J}^{(k)}(x)-\#\mathcal{B}_\mathcal{I,J,L}^{(k)}(x)}{x}
\Bigg|+2\varepsilon\label{eq:10131}
\end{align}
for sufficiently large $x$. 
Define 
\begin{equation}\label{eq:1015}
\mathcal{C}_\mathcal{I,J,L}^{(k)}:=\mathcal{B}_\mathcal{I,J,L}^{(k)}\setminus \mathcal{B}_\mathcal{I,J}^{(k)}=\{n\in \mathcal{B}_\mathcal{I,J,L}^{(k)}\mid (n^k,(n+2)^k)\cap \mathcal{S}_{\Lambda_k\setminus \mathcal{L}}\neq \varnothing\}
\end{equation}
and 
\[
\mathcal{C}_\mathcal{I,J,L}^{(k,\text{odd})}:=\{n\in \mathcal{C}_\mathcal{I,J,L}^{(k)}\mid \text{$n$\,:\,odd}\}.%\qquad \text{\rm and}\qquad 
\]
Let $x>2$ be a real number. 
Then  
the map 
$\rho:\mathcal{C}_\mathcal{I,J,L}^{(k,\text{odd})}(x)\to \mathcal{S}_{\Lambda_k\setminus \mathcal{L}}((x+2)^k)$ defined by 
\[
\rho(n):=\min\{m\in\mathbb{Z}_{\ge1}\mid m\in (n^k,(n+2)^k)\cap \mathcal{S}_{\Lambda_k\setminus \mathcal{L}}\}
\]
is well-defined from \eqref{eq:1015}, and moreover, it is injective 
since 
\[
\rho(n_j)\in(n_j^k,(n_j+2)^k) \quad (j=1,2)\qquad\text{and}\qquad 
(n_1^k,(n_1+2)^k)\cap(n_2^k,(n_2+2)^k)=\varnothing
\] for any distinct odd integers $n_1,n_2\in\mathcal{C}_\mathcal{I,J,L}^{(k,\text{odd})}(x)$. 
Hence, by property \eqref{prop2}, we have 
\[
\#\mathcal{C}_\mathcal{I,J,L}^{(k,\text{odd})}(x)\leq \#\mathcal{S}_{\Lambda_k\setminus \mathcal{L}}((x+2)^k)\le \sum_{\substack{z\in\mathbb{Z}_{\ge1},\,\lambda\in\Lambda_k\setminus \mathcal{L},\\z^k\lambda^k<(x+2)^k}}1\le 
\sum_{\lambda\in \Lambda_k\setminus \mathcal{L}}^{}\frac{x+2}{\lambda}<\varepsilon(x+2)<2\varepsilon x.
\]
Similarly, we can obtain the same upper bound for $\mathcal{C}_\mathcal{I,J,L}^{(k,\text{even})}:=
\mathcal{C}_\mathcal{I,J,L}^{(k)}\setminus\mathcal{C}_\mathcal{I,J,L}^{(k,\text{odd})}$, and so by \eqref{eq:1015}
\begin{equation}\label{A3}
0\le \#\mathcal{B}_\mathcal{I,J,L}^{(k)}(x)-\#\mathcal{B}_\mathcal{I,J}^{(k)}(x)=\#\mathcal{C}_\mathcal{I,J,L}^{(k)}(x)=
\#\mathcal{C}_\mathcal{I,J,L}^{(k,\text{odd})}(x)+\#\mathcal{C}_\mathcal{I,J,L}^{(k,\text{even})}(x)<4\varepsilon x.
\end{equation}
Therefore, by \eqref{eq:10131} and \eqref{A3}, we obtain 
\[
\Bigg|
\frac{
\#\mathcal{B}_\mathcal{I,J}^{(k)}(x)
}{x}
-d_\mathcal{I,J}^{(k)}
\Bigg|<6\varepsilon
\]
for sufficiently large $x$, and the proof of Theorem~\ref{thm2} is completed.

%%%%%%%%%%%%%%%%%%%%%%%%%%%%%%%%%%%%%%%%%%%%%%%%%%%%%%%%%%%%%%%%%%%%%%%%%%%%%%%%%%%%%%%%%%%%%%

\section{Proof of Theorem~\ref{thm1}}\label{sec4}

%We first show the following lemma. 
Let $\mathcal{A}_{\ell,m}^{(k)}$ and $\mathcal{B}_\mathcal{I,J}^{(k)}$ be 
defined as in \eqref{eq:1024} and \eqref{AIJ}, respectively.

\begin{lemma}\label{lem:1024}
For any integers $\ell,m\ge0$, we have 
% be integers. 
%Then the set can be written as the disjoint union 
\[
\mathcal{A}_{\ell,m}^{(k)}=
\bigsqcup_{\substack{\mathcal{I}, \mathcal{J}\subseteq \Lambda_k\\\#\mathcal{I}=
\ell,\,\#\mathcal{J}=m,\,\mathcal{I}\cap\mathcal{J}=\varnothing}}
\mathcal{B}_{\mathcal{I,J}}^{(k)}.
\]
\end{lemma}
%%%%%%%%%%%%%%%%%%%%%%%%%%%%%%
\begin{proof}
Clearly $\mathcal{A}_{\ell,m}^{(k)}\supseteq \mathcal{B}_\mathcal{I,J}^{(k)}$ for any $\mathcal{I,J}\subseteq \Lambda_k$ with $\#\mathcal{I}=
\ell,\,\#\mathcal{J}=m,\,\mathcal{I}\cap\mathcal{J}=\varnothing$. 
Let $n\in\mathcal{A}_{\ell,m}^{(k)}$. 
Since $\mathcal{S}_{k}=\sqcup_{\lambda\in\Lambda_k}\mathcal{S}_{\{\lambda\}}$, we have
\[
\ell=\#\bigl(
(n^k,(n+1)^k)\cap \mathcal{S}_{k}
\bigr)
=\sum_{\lambda\in \Lambda_k}\#
\bigl(
(n^k,(n+1)^k)\cap \mathcal{S}_{\{\lambda\}}
\bigr),
\]
and hence, from the equivalence of \eqref{s1} and \eqref{s3} in Lemma~\ref{lem:1018}, there exists a unique subset
$\mathcal{I}\subseteq\Lambda_k$ with $\#\mathcal{I}=\ell$ such that 
\[
\left\{
\begin{array}{l}
\text{$\#
\bigl(
(n^k,(n+1)^k)\cap \mathcal{S}_{\{\lambda\}}
\bigr)=1$ for any $\lambda\in \mathcal{I}$},\\[0.05cm]
\text{$
(n^k,(n+1)^k)\cap \mathcal{S}_{\{\lambda\}}
=\varnothing$ for any $\lambda\in \Lambda_k\setminus\mathcal{I}$}.
\end{array}
\right.
\]
Thus, noting $\mathcal{S}_\mathcal{I}=\sqcup_{\lambda\in\mathcal{I}}
\mathcal{S}_{\{\lambda\}}$, we obtain 
\begin{align*}
\#
\bigl(
(n^k,(n+1)^k)\cap \mathcal{S}_{\mathcal{I}}
\bigr)&=\sum_{\lambda\in \mathcal{I}}\#
\bigl(
(n^k,(n+1)^k)\cap \mathcal{S}_{\{\lambda\}}
\bigr)=\#\mathcal{I},\\
(n^k,(n+1)^k)\cap \mathcal{S}_{\Lambda_k\setminus\mathcal{I}}
&=
\bigcup_{\lambda\in \Lambda_k\setminus\mathcal{I}}
\bigl((n^k,(n+1)^k)\cap \mathcal{S}_{\{\lambda\}}
\bigr)=\varnothing.
\end{align*}
Similarly, there exists a unique subset 
$\mathcal{J}\subseteq\Lambda_k$ with $\#\mathcal{J}=m$ such that 
\[
\#
\bigl(
((n+1)^k,(n+2)^k)\cap \mathcal{S}_{\mathcal{J}}
\bigr)=\#\mathcal{J}\qquad \text{and}\qquad 
((n+1)^k,(n+2)^k)\cap \mathcal{S}_{\Lambda_k\setminus\mathcal{J}}
=\varnothing.
\]
Therefore, there exists a unique pair of subsets $\mathcal{I,J}\subseteq\Lambda_k$ such that $n\in\mathcal{B}_\mathcal{I,J}^{(k)}$. 
This completes the proof of Lemma \ref{lem:1024}. 
\end{proof}

%%%%%%%%%%%%%%%%%%%%%%%%%%%%%%%%%%%%%%%%%%%%%%%%%%%%%%%
Theorem~\ref{thm1} follows from Theorem~\ref{thm2} and Lemma~\ref{lem:1024}. 

\begin{proof}[Proof of Theorem~\ref{thm1}]
The proof is similar to that of Theorem~\ref{thm2}. 
Let $\Lambda_k$ be as in \eqref{Lambdaset} and $\varepsilon>0$ be arbitrary. 
Then there exists the least positive integer $N=N(\varepsilon)(\ge \ell+m)$ such that the set 
$\mathcal{L}:=\{\lambda_1,\lambda_2,\dots,\lambda_N\}\subseteq\Lambda_k$ satisfies the following properties: 
\begin{enumerate}[(i)]

\item\label{pro22}
$\displaystyle\sum_{\lambda\in\Lambda_k\setminus \mathcal{L}}^{}\lambda^{-1}
<\varepsilon$.

\item\label{pro32}
$
0\le e_{\ell,m}^{(k)}-e_{\ell,m,\mathcal{L}}^{(k)}<
\varepsilon$,
where
\[
e_{\ell,m}^{(k)}:=\sum_{\substack{\mathcal{I}, \mathcal{J}\subseteq \Lambda_k\\\#\mathcal{I}=\ell,\,\#\mathcal{J}=m,\,\mathcal{I}\cap\mathcal{J}=\varnothing}}d(\mathcal{B}_{\mathcal{I,J}}^{(k)}),\qquad 
e_{\ell,m,\mathcal{L}}^{(k)}:=\sum_{\substack{\mathcal{I}, \mathcal{J}\subseteq \Lambda_k\\\#\mathcal{I}=\ell,\,\#\mathcal{J}=m,\,\mathcal{I}\cap\mathcal{J}=\varnothing\\
\mathcal{I}\cup\mathcal{J}\subseteq \mathcal{L}}}d(\mathcal{B}_{\mathcal{I,J}}^{(k)}).
\]
\end{enumerate}
%We define the subset $\mathcal{B}_{\ell,m,\mathcal{L}}^{(k)}$ of $\mathcal{A}_{\ell,m}^{(k)}$ by 
Define 
\begin{equation}\label{finiteness}
\mathcal{A}_{\ell,m,\mathcal{L}}^{(k)}:=
\bigsqcup_{\substack{\mathcal{I}, \mathcal{J}\subseteq \Lambda_k\\\#\mathcal{I}=\ell,\,\#\mathcal{J}=m,\,\mathcal{I}\cap\mathcal{J}=\varnothing\\
\mathcal{I}\cup\mathcal{J}\subseteq \mathcal{L}}}\mathcal{B}_{\mathcal{I,J}}^{(k)}.
\end{equation}
Since $\mathcal{L}$ is finite, the right-hand side of \eqref{finiteness} is a finite union.  
Hence, by Theorem~\ref{thm2}, we have 
\begin{equation}\label{useth1}
\Bigg|
\frac{\#\mathcal{A}_{\ell,m,\mathcal{L}}^{(k)}(x)}{x}-e_{\ell,m,\mathcal{L}}^{(k)}
\Bigg|\leq 
\sum_{\substack{\mathcal{I}, \mathcal{J}\subseteq \Lambda_k\\\#\mathcal{I}=\ell,\,\#\mathcal{J}=m,\,\mathcal{I}\cap\mathcal{J}=\varnothing\\
\mathcal{I}\cup\mathcal{J}\subseteq \mathcal{L}}}
\Bigg|
\frac{\#\mathcal{B}_\mathcal{I,J}^{(k)}(x)}{x}-d(\mathcal{B}_\mathcal{I,J}^{(k)})
\Bigg|<\varepsilon
\end{equation}
for sufficiently large $x$. 
%$\mathcal{A}_\mathcal{I,J,K}$ be a subset of positive integers defined by
Thus, by \eqref{useth1} and property \eqref{pro32}, we obtain   
\begin{align}
\Bigg|
\frac{\#\mathcal{A}_{\ell,m}^{(k)}(x)}{x}-e_{\ell,m}^{(k)}
\Bigg|&\leq 
\Bigg|
\frac{\#\mathcal{A}_{\ell,m}^{(k)}(x)}{x}-\frac{\#\mathcal{A}_{\ell,m,\mathcal{L}}^{(k)}(x)}{x}
\Bigg|+
\Bigg|
\frac{\#\mathcal{A}_{\ell,m,\mathcal{L}}^{(k)}(x)}{x}-e_{\ell,m,\mathcal{L}}^{(k)}
\Bigg|
+\big|e_{\ell,m,\mathcal{L}}^{(k)}-e_{\ell,m}^{(k)}\big|\nonumber\\
&<\Bigg|
\frac{\#\mathcal{A}_{\ell,m}^{(k)}(x)-\#\mathcal{A}_{\ell,m,\mathcal{L}}^{(k)}(x)}{x}
\Bigg|
+2\varepsilon\label{eq:1013144}
\end{align}
for sufficiently large $x$. 

On the other hand, by Lemma~\ref{lem:1024}, we have 
%Let  Note that $\mathcal{A}_{\ell,m,\mathcal{L}}^{(k)}\subseteq\mathcal{A}_{\ell,m}^{(k)}$ 
\begin{equation}\label{eq:1015333}
\mathcal{C}_{\ell,m,\mathcal{L}}^{(k)}:=\mathcal{A}_{\ell,m}^{(k)}\setminus \mathcal{A}_{\ell,m,\mathcal{L}}^{(k)}=
\bigsqcup_{\substack{\mathcal{I}, \mathcal{J}\subseteq \Lambda_k\\\#\mathcal{I}=\ell,\,\#\mathcal{J}=m,\,\mathcal{I}\cap\mathcal{J}=\varnothing\\
\mathcal{I}\cup\mathcal{J}\nsubseteq \mathcal{L}}}\mathcal{B}_{\mathcal{I,J}}^{(k)}.
\end{equation}
%Then, similarly to the proof of Theorem~\ref{thm2}, we 
Let $x>2$ be a real number and define the map 
%\[
%\mathcal{C}_{\ell,m,\mathcal{L}}^{(\text{\rm odd})}:=\{n\in \mathcal{C}_{\ell,m,\mathcal{L}}\mid n:\text{\rm odd}\},%\qquad \text{\rm and}\qquad 
%\mathcal{A}_{I,J,K}^{(2,\text{\rm even})}:=\{n\in \mathcal{A}_{I,J,K}^{(2)}\mid n:\text{\rm even}\}.
%\]
%consider the map 
\[
\psi:\mathcal{C}_{\ell,m,\mathcal{L}}^{(k,\text{\rm odd})}(x):=
\{n\in \mathcal{C}_{\ell,m,\mathcal{L}}^{(k)}\mid n:\text{\rm odd}, n\le x\}\to 
\mathcal{S}_{\Lambda_k\setminus \mathcal{L}}\big((x+2)^k\big)\] 
by 
\[
\psi(n):=\min\{m\in\mathbb{Z}_{\ge1}\mid m\in (n^k,(n+2)^k)\cap \mathcal{S}_{\Lambda_k\setminus \mathcal{L}}\}.
\]
%if $\mathcal{C}_{\ell,m,\mathcal{L}}^{(k,\text{\rm odd})}(x)$ is non-empty. 
Then $\psi$ is well-defined. 
Indeed, if $n\in\mathcal{C}_{\ell,m,\mathcal{L}}^{(k,\text{\rm odd})}(x)$, then, by \eqref{eq:1015333}, 
there exist subsets $\mathcal{I,J}\subseteq \Lambda_k$ such that 
$\mathcal{I}\cup \mathcal{J}\nsubseteq \mathcal{L}$ and  $n\in \mathcal{B}_\mathcal{I,J}^{(k)}$. 
Hence, there exists a $\lambda\in \mathcal{I}\cup\mathcal{J}$ with $\lambda\notin \mathcal{L}$ 
such that 
\[
\varnothing\neq (n^k,(n+2)^k)\cap\mathcal{S}_{\{\lambda\}}
\subseteq  (n^k,(n+2)^k)\cap\mathcal{S}_{\Lambda_k\setminus\mathcal{L}}
\subseteq \mathcal{S}_{\Lambda_k\setminus \mathcal{L}}((x+2)^k).
%\supseteq (n^k,(n+2)^k)\cap\mathcal{S}_{\Lambda_k\setminus\mathcal{L}}\supseteq  (n^k,(n+2)^k)\cap\mathcal{S}_{\{\lambda\}} \neq \varnothing. 
\]
Moreover, similarly to the proof of Theorem~\ref{thm2}, 
we find that $\psi$ is injective and   
\[
\#\mathcal{C}_{\ell,m,\mathcal{L}}^{(k,\text{\rm odd})}(x)\leq \#\mathcal{S}_{\Lambda_k\setminus \mathcal{L}}((x+2)^k)<2\varepsilon x
\]
as well as the same upper bound for 
$\mathcal{C}_{\ell,m,\mathcal{L}}^{(k,\text{\rm even})}$. %(x):=
%\{n\in \mathcal{C}_{\ell,m,\mathcal{L}}^{(k)}\mid n:\text{\rm even}, n\le x\}$. 
Thus, by \eqref{eq:1015333}, we have  
\begin{equation}\label{A31}
0\le \#\mathcal{A}_{\ell,m}^{(k)}(x)-\#\mathcal{A}_{\ell,m,\mathcal{L}}^{(k)}(x)=\#\mathcal{C}_{\ell,m,\mathcal{L}}^{(k)}(x)=
\#\mathcal{C}_{\ell,m,\mathcal{L}}^{(k,\text{\rm odd})}(x)+\#\mathcal{C}_{\ell,m,\mathcal{L}}^{(k,\text{\rm even})}(x)
<4\varepsilon x.
\end{equation}

Therefore, by \eqref{eq:1013144} and \eqref{A31}, we obtain
\[
\Bigg|
\frac{\#\mathcal{A}_{\ell,m}^{(k)}(x)}{x}-e_{\ell,m}^{(k)}
\Bigg|<6\varepsilon
\]
for sufficiently large $x$ and the proof of Theorem~\ref{thm1} is completed.  
\end{proof}

%%%%%%%%%%%%%%%%%%
%%%%%%%%%%%%%%%%%%%%

\section{Proof of Corollary~\ref{cor:1025}}\label{sec5}

%%%%%%%%%%%%%%%%%%%%
%%%%%%%%%%%%%%%%%%%%
%\begin{proof}[Proof of Corollary~\ref{cor:1025}]
Let $\ell\geq0$ be an integer, and let $\Lambda_k$ be as in \eqref{Lambdaset}. 
Let $\varepsilon>0$ be arbitrary. 
Since $\sum_{m\ge0}^{}d(\mathcal{A}_{\ell,m}^{(k)})$ converges by 
\eqref{eq:02061}, there exists the least positive 
integer $N=N(\varepsilon)\geq1$ satisfying 
\begin{equation}\label{eq:kome}
\sum_{j> N}^{}{\lambda_j}^{-1}<\varepsilon\qquad\text{and}\qquad  
\sum_{m> N}^{}d(\mathcal{A}_{\ell,m}^{(k)})<\varepsilon.
\end{equation}
Define 
$\mathcal{L}:=\{\lambda_1,\dots,\lambda_N\}$, $\mathcal{A}_{\ell,\le N}^{(k)}:=\bigsqcup_{m=0}^{N}\mathcal{A}_{\ell,m}^{(k)}$, and 
\[
%\mathcal{B}_{\ell,\le N}^{(k)}:=\bigsqcup_{m=0}^{N}\mathcal{A}_{\ell,m}^{(k)}\qquad \text{and}\qquad 
\mathcal{A}_{\ell,> N}^{(k)}:=\mathcal{A}_{\ell}^{(k)}\setminus \mathcal{A}_{\ell,\le N}^{(k)}=\bigsqcup_{m\ge N+1}^{}\mathcal{A}_{\ell,m}^{(k)}.
\]
By Theorem~\ref{thm1}, we have 
\begin{equation}\label{Bse}
\Bigg|
\frac{\#\mathcal{A}_{\ell,\le N}^{(k)}(x)}{x}-\sum_{m=0}^{N}d(\mathcal{A}_{\ell,m}^{(k)})
\Bigg|\leq 
\sum_{m=0}^{N}
\Bigg|
\frac{\#\mathcal{A}_{\ell,m}^{(k)}(x)}{x}-d(\mathcal{A}_{\ell,m}^{(k)})
\Bigg|<\varepsilon
\end{equation}
for sufficiently large $x$. 

%Let $x>2$ be a real number. 
On the other hand, for each integer $n\in\mathcal{A}_{\ell,> N}^{(k)}$, there exists an integer $m_0\ge N+1$ with 
$n\in\mathcal{A}_{\ell,m_0}^{(k)}$, and so 
\begin{equation}\label{eq:1122}
N+1\le m_0=\#(I_n\cap \mathcal{S}_{k})=\sum_{\lambda\in\Lambda_k}\#(I_n\cap \mathcal{S}_{\{\lambda\}})
\end{equation}
with $I_n:=\bigl(
(n+1)^k,(n+2)^k
\bigr)$. 
By Lemma~\ref{lem:1018}, we have $\#(I_n\cap \mathcal{S}_{\{\lambda\}})\le1$ for every $\lambda\in\Lambda_k$, and hence, it follows from \eqref{eq:1122} and $\#\mathcal{L}=N$  that there exists at least one $\lambda=\lambda(n)\in\Lambda_k\setminus\mathcal{L}$ satisfying  
$I_n\cap \mathcal{S}_{\{\lambda\}}\neq \varnothing$. 
Thus, considering the injective map $\sigma:\mathcal{A}_{\ell,> N}^{(k)}(x)\to \mathcal{S}_{\Lambda_k\setminus\mathcal{L}}((x+2)^k)$ defined by 
$\sigma(n):=\min\{m\in\mathbb{Z}_{\ge1}\mid m\in I_n\cap \mathcal{S}_{\Lambda_k\setminus\mathcal{L}}\}
$, 
we can obtain 
\begin{equation}\label{eq:est}
0\le \#\mathcal{A}_{\ell}^{(k)}(x)-\#\mathcal{A}_{\ell,\le N}^{(k)}(x)\le \#\mathcal{A}_{\ell,> N}^{(k)}(x)\le 2\varepsilon x
\end{equation}
%from \eqref{eq:kome} 
similarly to the proof of Theorem~\ref{thm2}. 
Therefore, by \eqref{eq:kome}, \eqref{Bse} and \eqref{eq:est}, we have 
\[
\Bigg|
\frac{\#\mathcal{A}_{\ell}^{(k)}(x)}{x}-\sum_{m\ge0}^{}d(\mathcal{A}_{\ell,m}^{(k)})
\Bigg|\leq 
\Bigg|
\frac{\#\mathcal{A}_{\ell}^{(k)}(x)-\#\mathcal{A}_{\ell,\le N}^{(k)}(x)}{x}
\Bigg|+2\varepsilon
%\Bigg|+
%\Bigg|
%\frac{\#\mathcal{B}_{\ell,\le N}^{(k)}(x)}{x}-\delta_{\ell,N}
%\Bigg|+
%(\delta_{\ell}-\delta_{\ell,N})
<4\varepsilon
\]
for sufficiently large $x$, and the proof is completed.
%\end{proof}
%%%%%%%%%%%%%%%%%%%%%%%%%%%%%%%%%%%%%%%%%%%%%%%%%%%%%%%%%%%%%%%%%%%%%%%%%%%%%%%%%%%%
\section{Explicit formulas and numerical values for $d(\mathcal{A}_{\ell,m}^{(k)})$}\label{sec6}

%In this section, we present some explicit formulas  
%for $d(\mathcal{A}_{\ell,m}^{(k)})$. % as an infinite sum. 
The infinite product 
\[
F_k(z):=\prod_{\lambda\in \Lambda_k}^{}\left(
1+\frac{z-2}{\lambda}
\right)
\]
is entire by \eqref{conv} and has a power series expansion $F_k(z)=\sum_{n=0}^{\infty}a_{k,n}z^n$ $(z\in\mathbb{C})$ with 
\begin{equation}\label{Tylor123}
a_{k,n}:=\sum_{\substack{\mathcal{L}\subseteq \Lambda_k\\\#\mathcal{L}=n}}^{}\prod_{\lambda\in\mathcal{L}}\frac{1}{\lambda}\cdot \prod_{\lambda\in\Lambda_k\setminus \mathcal{L}}^{}
\left(
1-\frac{2}{\lambda}
\right),\quad n\ge0.
\end{equation}
Hence, for any integers $\ell,m\ge0$, we have by \eqref{eq:1205} and \eqref{Tylor123}
%the asymptotic density $d(\mathcal{A}_{\ell,m}^{(k)})$ is expressed as    
\begin{equation}\label{eq:120501}
d(\mathcal{A}_{\ell,m}^{(k)})
=\binom{\ell+m}{\ell}a_{k,\ell+m},%=\frac{1}{\ell! m!}F_k^{(\ell+m)}(0). 
\end{equation}
and so, by \eqref{eq:1025444} %Corollary~\ref{cor:1025} 
and \eqref{eq:120501},  
\begin{align}
\sum_{\ell\ge 0}^{}d(\mathcal{A}_{\ell}^{(k)})z^\ell&=\sum_{\ell,m\ge0}^{}d(\mathcal{A}_{\ell,m}^{(k)})z^\ell=
\sum_{n\ge0}^{}\sum_{\ell=0}^{n}\binom{n}{\ell}a_{k,n}z^\ell=\sum_{n\ge 0}^{}a_{k,n}(z+1)^n\label{eq:12055}\\
&=F_k(z+1)=\prod_{\lambda\in\Lambda_k}^{}
\left(1+\frac{z-1}{\lambda}\right),\nonumber
\end{align}
which is the formula~\eqref{gene} of Xiong and Zaharescu. 
Moreover, substituting $z=w-1$ into \eqref{eq:12055} and comparing the coefficients on 
both sides using \eqref{eq:120501}, we obtain the inverse formula for \eqref{eq:1025444}:
%%%%%%%%%%%%%%%%%%%%
\begin{equation}\label{eq:1206}
d(\mathcal{A}_{\ell,m}^{(k)})=\sum_{n\ge 0}(-1)^{n}\binom{\ell+m+n}{\ell,m,n}
d(\mathcal{A}_{\ell+m+n}^{(k)}),
\end{equation}
where $\binom{\ell+m+n}{\ell,m,n}$ is a trinomial coefficient. 
%\%%%%%%%%%%%%%%%%%
Similarly, substituting $z=w-2$ into  
\[
F_k(z+2)=\prod_{\lambda\in \Lambda_k}^{}
\bigg(
1+\frac{z}{\lambda}
\bigg)=
\sum_{n\ge0}^{}\xi_{n}^{(k)}z^n,
\]
where $(\xi_{n}^{(k)})_{n\ge0}$ is a sequence defined in~\eqref{eq:112514}, 
we obtain  
\begin{equation}\label{eq:20251125}
d(\mathcal{A}_{\ell,m}^{(k)})=
\sum_{n\ge 0}(-2)^n\binom{\ell+m+n}{\ell,m,n}\xi_{\ell+m+n}^{(k)}.
\end{equation}
The first few numerical values of $d(\mathcal{A}_{\ell}^{(k)})$ and $d(\mathcal{A}_{\ell,m}^{(k)})$ for $k=2,3$ 
are presented in Tables~\ref{table0}-\ref{table2} below. 
These values were 
computed in Python using the explicit formulas~\eqref{eq:1124} and \eqref{eq:20251125}. 
On the other hand, substituting $z=1$ into \eqref{eq:12055} yields   
\begin{equation}\label{eq:02061}
\sum_{\ell\ge0}d(\mathcal{A}_{\ell}^{(k)})=\sum_{\ell,m\ge0}d(\mathcal{A}_{\ell,m}^{(k)})=1,
\end{equation}
which shows the countable additivity of the asymptotic density 
over the disjoint unions 
$
\mathbb{N}=\sqcup_{\ell\ge0}\mathcal{A}_{\ell}^{(k)}$ and 
$\mathbb{N}=\sqcup_{\ell,m\ge0}\mathcal{A}_{\ell,m}^{(k)}
$, respectively. 

The following theorem determines the maximum values of $d(\mathcal{A}_{\ell}^{(k)})$ and $d(\mathcal{A}_{\ell,m}^{(k)})$ for $k=2,3$. 
%%%%%%%%%%%%%%%%%%%%%%%%%%%%%%%%%%%%%%%%%%%%%%%%%%%%%%%%%%%
\begin{theorem}\label{max3}
The maximum values of $d(\mathcal{A}_{\ell}^{(k)})$ and $d(\mathcal{A}_{\ell,m}^{(k)})$ for $k=2,3$ are given by 
\begin{align*}
\max_{\ell\ge0}d(\mathcal{A}_{\ell}^{(2)})&=d(\mathcal{A}_{1}^{(2)})=0.395565\dots,
&\max_{\ell,m\ge0}d(\mathcal{A}_{\ell,m}^{(2)})=d(\mathcal{A}_{1,1}^{(2)})=0.158761\dots,\nonumber\\
\max_{\ell\ge0}d(\mathcal{A}_{\ell}^{(3)})&=d(\mathcal{A}_{3}^{(3)})=
0.220239\dots,
&\max_{\ell,m\ge0}d(\mathcal{A}_{\ell,m}^{(3)})=d(\mathcal{A}_{3,3}^{(3)})=0.048348\dots.
\end{align*}
\end{theorem}
%%%%%%%%%%%%%%%%%%%%%%%%%%%%%%%%%%%%%%%%%%%%%%%%%%%%%%%%%%
\begin{proof}
Combining \eqref{eq:02061} with Shiu's estimates 
$d(\mathcal{A}_{0}^{(2)})=0.275\dots$ and 
$d(\mathcal{A}_{1}^{(2)})=0.395\dots$ 
(cf. \cite[p.~176]{S80}; see also Table~\ref{table0} below), 
we obtain  
\[
\max_{\ell\ge 2}d(\mathcal{A}_{\ell}^{(2)})<
\sum_{\ell\ge2}d(\mathcal{A}_{\ell}^{(2)})=
1-d(\mathcal{A}_{0}^{(2)})-d(\mathcal{A}_{1}^{(2)})<0.33,
\]
which shows $\max_{\ell\ge0}d(\mathcal{A}_{\ell}^{(2)})=d(\mathcal{A}_{1}^{(2)})$. 
Similarly, it follows from Table~\ref{table1} and \eqref{eq:02061} that 
\[
\max_{0\le \ell,m\le 3}d(\mathcal{A}_{\ell,m}^{(2)})=
d(\mathcal{A}_{1,1}^{(2)})=0.158\dots
\]
and
\[
\max_{\max(\ell,m)\ge4}
d(\mathcal{A}_{\ell,m}^{(2)})<
\sum_{\max(\ell,m)\ge4}d(\mathcal{A}_{\ell,m}^{(2)})=
1-
\sum_{0\le \ell,m\le 3}
d(\mathcal{A}_{\ell,m}^{(2)})< 0.040074,
\]
and hence, we have $\max_{\ell,m\ge0}d(\mathcal{A}_{\ell,m}^{(2)})=d(\mathcal{A}_{1,1}^{(2)})$. 
Moreover, combining Tables~\ref{table0} and \ref{table2} with \eqref{eq:1025444} and \eqref{eq:02061}, 
we obtain the latter assertions since 
\begin{align*}
\max_{\ell \ge 6} d(\mathcal{A}_{\ell}^{(3)}) &< \sum_{\ell \ge 6} d(\mathcal{A}_{\ell}^{(3)}) = 1 - \sum_{0 \le \ell \le 5} d(\mathcal{A}_{\ell}^{(3)}) < 0.16 \\
&< 0.220\ldots = d(\mathcal{A}_{3}^{(3)}) = \max_{0 \le \ell \le 5} d(\mathcal{A}_{\ell}^{(3)})
\end{align*}
and
\begin{align*}
\max_{\max(\ell,m) \ge 7} d(\mathcal{A}_{\ell,m}^{(3)}) &< \max_{\ell \ge 7} d(\mathcal{A}_{\ell}^{(3)}) \le 
\max \biggl\{ d(\mathcal{A}_{7}^{(3)}), \,\sum_{\ell \ge 8} d(\mathcal{A}_{\ell}^{(3)}) \biggr\} \\
&= \max \biggl\{ d(\mathcal{A}_{7}^{(3)}), \,1 - \sum_{0 \le \ell \le 7} d(\mathcal{A}_{\ell}^{(3)}) \biggr\} 
\le \max \{ 0.042, 0.026 \} \\
&< 0.048\ldots = d(\mathcal{A}_{3,3}^{(3)}) = \max_{0 \le \ell,m \le 6} d(\mathcal{A}_{\ell,m}^{(3)}).
\end{align*}
The proof of Theorem~\ref{max3} is completed.
\end{proof}
%%%%%%%%%%%%%%%%%%%%%%%%%%%%%%%%%%%%%%%%%%%%%%%%%
Extending the result of Theorem~\ref{max3} to the case $k \ge 4$ is not straightforward, 
as the current proof relies heavily on numerical computations.  
It is conjectured that the indices $\ell$ and $m$ maximizing 
the 
densities $d(\mathcal{A}_{\ell}^{(k)})$ and $d(\mathcal{A}_{\ell,m}^{(k)})$ 
increase as $k$ grows; this dependency significantly complicates the general case. 
The problem of determining these maximum densities for $k \ge 4$ remains to be explored.

\newpage

%\newpage
%Determining these maximum densities for $k \ge 3$ therefore 
%remains a problem.   
%remains an interesting open problem. %
%would be interesting problem.
%The proof of Theorem~\ref{max3} relies heavily on numerical computations; consequently, extending this result to $k \ge 3$ is not easy.
%The determination of the maximum values of $d(\mathcal{A}_{\ell}^{(k)})$ and $d(\mathcal{A}_{\ell,m}^{(k)})$ for $k \ge 3$ remains an interestin problem.

%\begin{center}
%  \text{Tables: Numerical values of 
%$d(\mathcal{A}_{\ell,m}^{(k)})$ $(k=2,3,0\le \ell,m\le5)$ 
%truncated to six decimal places} %
%\end{center}
%\subsubsection*{Tables}
%\appendix % 
%\subsection{Tables: Numerical values of 
%$d(\mathcal{A}_{\ell,m}^{(k)})$ $(k=2,3,0\le \ell,m\le5)$  \\
%truncated to six decimal places}\label{App1}

%\vspace{0.2cm}

%\par\bigskip 
{\centering Tables. The first few numerical values of $d(\mathcal{A}_{\ell}^{(k)})$ and 
$d(\mathcal{A}_{\ell, m}^{(k)})$ for $k=2,3$,\\ 
truncated to six decimal places \par} % 
%\nopagebreak \medskip % 
\begin{table}[H]
    \centering
    \caption{$d(\mathcal{A}_{\ell}^{(2)})$ and $d(\mathcal{A}_{\ell}^{(3)})$}
    \vspace{2mm}
    \label{table0}    
    \begin{tabular}{r|rr}
        $\ell$ & $d(\mathcal{A}_{\ell}^{(2)})$ & $d(\mathcal{A}_{\ell}^{(3)})$ \\[0.5ex] \hline 
        \rule{0pt}{2.5ex} 0 & 0.275965 & 0.020037 \\
        1 & 0.395565 & 0.084806 \\
        2 & 0.231299 & 0.171014 \\
        3 & 0.077074 & 0.220239 \\
        4 & 0.017015 & 0.204704 \\
        5 & 0.002714 & 0.147035 \\
        6 & 0.000331 & 0.085293 \\
        7 & 0.000032 & 0.041214 \\
    \end{tabular}
\end{table}
%\begin{table}[H]
%    \centering
%    \caption{$d(\mathcal{A}_{\ell}^{(2)})$ and $d(\mathcal{A}_{\ell}^{(3)})$}
%    \vspace{2mm}
%    \label{table0}    
%    \begin{tabular}{r|rrrrrrrr}
%\scriptsize{\diagbox{}{$\ell$}}&$0$ & $1$ & $2$ & $3$ & $4$ & $5$ & $6$ & 7\\\hline\rule{0pt}{3ex}
%        $d(\mathcal{A}_{\ell}^{(2)})$ & 0.275965 & 0.395565 & 0.231299 & 0.077074 & 0.017015 & 0.002714 & 0.000331&0.000032\\[0.1em]
%        $d(\mathcal{A}_{\ell}^{(3)})$ &  0.020037& 0.084806 & 0.171014 & 0.220239 & 0.204704 & 0.147035 & 0.085293&0.041214\\
%        \end{tabular}
%\end{table}
\begin{table}[H]
    \centering
    \caption{$d(\mathcal{A}_{\ell,m}^{(2)})$}
    \vspace{2mm}
    \label{table1}
    \begin{tabular}{r|rrrrrr}
\scriptsize{\diagbox{$\ell$}{$m$}}& $0$ & $1$ & $2$ & $3$ & $4$ & $5$ \\ \hline \rule{0pt}{2.5ex} 
        $0$ & 0.049227 & 0.107920 & 0.079380& 0.030530& 0.007444& 0.001278\\
        $1$ &  & 0.158761 & 0.091591& 0.029777& 0.006393& 0.000991\\
        $2$ &  & & 0.044666& 0.012786& 0.002478& 0.000352\\
        $3$ &  &  & & 0.003304& 0.000588& 0.000077\\
        $4$ &  &  & & & 0.000097& 0.000012\\
        $5$ &  &  & & & & 0.000001\\
\end{tabular}
\end{table}
\begin{table}[H]
    \centering
    \caption{$d(\mathcal{A}_{\ell,m}^{(3)})$}
      \vspace{2mm}
    \label{table2}
    \begin{tabular}{r|rrrrrrr}
\scriptsize{\diagbox{$\ell$}{$m$}}& $0$ & $1$ & $2$ & $3$ & $4$ & $5$ & $6$\\\hline \rule{0pt}{2.5ex} 
        $0$ & 0.000146 & 0.000898 & 0.002413& 0.003899& 0.004360& 0.003654&0.002417\\
        $1$ &  & 0.004826 & 0.011698& 0.017443& 0.018274& 0.014504&0.009157\\
        $2$ &  & & 0.026165& 0.036549& 0.036261& 0.027472&0.016659\\
        $3$ &  &  & & 0.048348& 0.045787& 0.033318&0.019498\\
        $4$ &  &  & & & 0.041647& 0.029247&0.016580\\
        $5$ &  &  & & & & 0.019896&0.010961\\
        $6$ &  &  & & & & &0.005883\\      
        
    \end{tabular}
\end{table}

\section*{Acknowledgments}
This work was supported by JSPS KAKENHI Grant Numbers JP25K06911.

\end{document}